\documentclass[10pt]{article}
\usepackage{amsmath, amsthm}\usepackage{enumerate}
\usepackage{amssymb}\usepackage{bbold}
\usepackage{color}
\newtheorem{theorem}{Theorem}[section]
\newtheorem{proposition}[theorem]{Proposition}

\newtheorem{corollary}[theorem]{Corollary}

\newtheorem{definition}[theorem]{Definition}

\DeclareMathOperator{\convw}{\xrightarrow[]{w}}

\DeclareMathOperator{\convo}{\xrightarrow[]{o}}
\DeclareMathOperator{\convtau}{\xrightarrow[]{\tau}}
\DeclareMathOperator{\convxi}{\xrightarrow[]{\xi}}
\DeclareMathOperator{\convctau}{\xrightarrow[]{c-\tau}}
\DeclareMathOperator{\convco}{\xrightarrow[]{c-o}}
\DeclareMathOperator{\convr}{\xrightarrow[]{ru}}
\DeclareMathOperator{\convcr}{\xrightarrow[]{c-ru}}

\large
\makeatletter
\renewcommand{\subsection}{\@startsection{subsection}{1}
{0pt}{3.25ex plus 1ex minus.2ex}{-1em}{\normalfont\normalsize\bf}}
\makeatother
\begin{document}

\title{{\bf Automatic boundedness of some operators between ordered and topological vector spaces}}
\maketitle
\author{\centering{{Eduard Emelyanov$^{1}$\\ 
\small $1$ Sobolev Institute of Mathematics, Novosibirsk, Russia}
\abstract{We study topological boundedness of order-to-topology bounded and order-to-topology continuous operators
from ordered vector spaces to topological vector spaces. The uniform boundedness principle for such operators is investigated.}

\vspace{3mm}

\vspace{3mm}
}}

\section{Introduction}

\hspace{4mm}
Order-to-topology continuous and order-to-norm bounded operators have been recently studied by different authors
in \cite{AEG2022,E0-2025,E1-2025,EEG2025,JAM2021,KTA2025,ZSC2023}. We extend some of their results to the setting 
of topological vector spaces and present conditions providing topological boundedness of such operators.

We abbreviate a normed (a topological, an ordered, an ordered normed, an ordered Banach) vector space by \text{\rm NS} 
(\text{\rm TVS}, \text{\rm OVS}, \text{\rm ONS}, \text{\rm OBS}).
In what follows, vector spaces are real, operators are linear, ${\cal L}(X,Y)$ stands for the space of operators from 
a vector space $X$ to a vector space $Y$, $B_X$ for the closed unit ball of a NS $X$, and $x_\alpha\downarrow 0$ for a decreasing net in an OVS 
such that $\inf\limits_\alpha x_\alpha=0$. A net $(x_\alpha)$ in an OVS $X$
\begin{enumerate}[-]
\item\ 
{\em order converges} to $x$ (o-converges to $x$, or $x_\alpha\convo x$) if there exists 
$g_\beta\downarrow 0$ in $X$ such that, for each $\beta$ there is $\alpha_\beta$ 
such that $\pm(x_\alpha-x)\le g_\beta$ for $\alpha\ge\alpha_\beta$. 
\item\ 
{\em relative uniform converges} to $x$ (ru-converges to $x$, or $x_\alpha\convr x$) 
if, for some $u\in X_+$ there exists an increasing sequence 
$(\alpha_n)$ of indices with $\pm(x_\alpha-x)\le\frac{1}{n}u$ for $\alpha\ge\alpha_n$. 
\end{enumerate}

We shall work with the following classes of operators. 

\begin{definition}\label{def 1}
An operator $T$ from an OVS $X$ to an OVS $Y$ is
\begin{enumerate}[$a)$]
\item\
{\em order bounded} if $T[a,b]$ is order bounded for every order interval $[a,b]$ in $X$
$($the set of such operators is a vector space which is denoted by ${\cal L}_{ob}(X,Y))$.
\item\
{\em order continuous} if $Tx_\alpha\convo 0$ in $Y$ whenever $x_\alpha\convo 0$
$($shortly, $T\in{\cal L}_{oc}(X,Y))$.
\item\
{\em ru-continuous} if $Tx_\alpha\convr 0$ in $Y$ whenever $x_\alpha\convr 0$ $($shortly, $T\in{\cal L}_{rc}(X,Y))$.
\end{enumerate}
An operator $T$ from an OVS $X$ to a TVS $(Y,\tau)$ is 
\begin{enumerate}
\item[$c)$]\
{\em order-to-topology continuous} if $Tx_\alpha\convtau 0$ whenever $x_\alpha\convo 0$ $($shortly, $T\in{\cal L}_{o{\tau}c}(X,Y))$. 
The set of such operators from an OVS $X$ to NS $Y$ is denoted by ${\cal L}_{onc}(X,Y)$. 
\item[$d)$]\
{\em ru-to-topology continuous} if $Tx_\alpha\convtau 0$ whenever $x_\alpha\convr 0$ $($shortly, $T\in{\cal L}_{r{\tau}c}(X,Y))$. 
The set of such operators from an OVS $X$ to NS $Y$ is denoted by ${\cal L}_{rnc}(X,Y)$. 
\item[$e)$]\
{\em order-to-topology bounded} if $T[a,b]$ is $\tau$-bounded for every $[a,b]\subseteq X$ $($shortly, $T\in{\cal L}_{o{\tau}b}(X,Y))$. 
The set of such operators from an OVS $X$ to a NS $Y$ is denoted by ${\cal L}_{onb}(X,Y)$. 
\end{enumerate}
The $\sigma$-versions of the above notions are obtained via replacing nets by sequences.
\end{definition}

We recall several definitions concerning collective convergences and collectively qualified sets of operators 
(see, \cite{AEG2024,E0a-2025,E0b-2025,E0-2025,E1-2025,EEG2025}).
Let ${\cal B}=\{(x^b_\alpha)_{\alpha\in A}\}_{b\in B}$ be a family of nets in a vector space $X$ which are indexed by a directed set $A$. If $X$ is an OVS, 
\begin{enumerate}[-]
\item\ 
${\cal B}$ {\em collective \text{\rm ru}-converges} to $0$ (briefly, ${\cal B}\convcr 0$)
if, for some $u\in X_+$ there exists an increasing sequence $(\alpha_n)$ of indices with $\pm x^b_\alpha\le\frac{1}{n}u$ for all
$\alpha\ge\alpha_n$ and $b\in B$. 
\item\ 
${\cal B}$ {\em collective \text{\rm o}-converges} to $0$ (briefly, ${\cal B}\convco 0$)
if there exists a net $(g_\beta)$, $g_\beta\downarrow 0$ in $X$ such that, for each $\beta$ there is $\alpha_\beta$
satisfying $\pm x^b_\alpha\le g_\beta$ for all $\alpha\ge\alpha_\beta$ and $b\in B$. 
\end{enumerate}
If $X=(X,\tau)$ is a TVS, 
\begin{enumerate}[-]
\item\ 
${\cal B}$ {\em collective $\tau$-converges} to $0$ (briefly, ${\cal B}\convctau 0$)
if, for each $U\in\tau(0)$ there exists $\alpha_U$ with $x^b_\alpha\in U$ for all $\alpha\ge\alpha_U$ and $b\in B$. 
\end{enumerate}

\medskip
\noindent
We shall use the following collective version of Definition \ref{def 1}.

\begin{definition}\label{def 2}
A set ${\cal T}$ of operators from an OVS $X$ to an OVS $Y$ is 
\begin{enumerate}[$a)$]
\item\ 
{\em collectively order bounded} $($${\cal T}\in\text{\bf L}_{ob}(X,Y)$$)$ if 
${\cal T}[a,b]$  is order bounded for every $[a,b]\subseteq X$. 
\item\ 
{\em collectively \text{\rm o}-continuous} $({\cal T}\in\text{\bf L}_{oc}(X,Y)$$)$ if 
${\cal T}x_\alpha\convco 0$ whenever $x_\alpha\convo 0$. 
\item\ 
{\em collectively \text{\rm ru}-continuous} $({\cal T}\in\text{\bf L}_{rc}(X,Y)$$)$ if 
${\cal T}x_\alpha\convcr 0$ whenever $x_\alpha\convr 0$. 
\end{enumerate}
A set ${\cal T}$ of operators from an OVS $X$ to a TVS $(Y,\tau)$ is 
\begin{enumerate}
\item[$c)$]\
{\em collectively order-to-topology continuous} $({\cal T}\in\text{\bf L}_{o{\tau}c}(X,Y)$$)$ if 
${\cal T}x_\alpha\convtau 0$ whenever $x_\alpha\convo 0$. 
The family of collectively order-to-norm continuous sets of operators from an OVS $X$ to a NS $Y$
is denoted by $\text{\bf L}_{onc}(X,Y)$.
\item[$d)$]\
{\em collectively \text{\rm ru}-to-topology continuous} $({\cal T}\in\text{\bf L}_{r{\tau}c}(X,Y)$$)$ if 
${\cal T}x_\alpha\convctau 0$ whenever $x_\alpha\convr 0$. 
The family of collectively ru-to-norm continuous sets of operators from an OVS $X$ to a NS $Y$
is denoted by $\text{\bf L}_{rnc}(X,Y)$.
\item[$e)$]\
{\em collectively order-to-topology bounded} $({\cal T}\in\text{\bf L}_{o{\tau}b}(X,Y)$$)$ if 
${\cal T}[a,b]$ is $\tau$-bounded for every $[a,b]\subseteq X$. 
The family of collectively order-to-norm bounded sets of operators from an OVS $X$ to a NS $Y$
is denoted by $\text{\bf L}_{onb}(X,Y)$.
\end{enumerate}
A set ${\cal T}$ of operators from a TVS $(X,\xi)$ to a TVS $(Y,\tau)$ is 
\begin{enumerate}[-]
\item\ 
{\em collectively continuous} $({\cal T}\in\text{\bf L}_{c}(X,Y)$$)$ if 
${\cal T}x_\alpha\convctau 0$ whenever $x_\alpha\convxi 0$. 
\item\ 
{\em collectively bounded} $({\cal T}\in\text{\bf L}_{b}(X,Y)$$)$ if 
${\cal T}U$ is $\tau$-bounded whenever $U$ is $\xi$-bounded. 
\end{enumerate}
The $\sigma$-versions of the above notions are obtained via replacing nets by sequences.
\end{definition}

The present note is organized as follows. We prove Theorem \ref{theorem 2}, which asserts 
collective \text{\rm ru}-to-topology continuity of collectively order-to-topology bounded sets.
In Theorem \ref{theorem 2a}, conditions for the inclusion $\text{\bf L}_{o{\tau}c}(X,Y)\subseteq\text{\bf L}_{o{\tau}b}(X,Y)$ are given.
Theorem \ref{theorem 1} gives conditions for the inclusion $\text{\bf L}_{o{\tau}b}(X,Y)\subseteq\text{\bf L}_{n{\tau}b}(X,Y)$.

For the terminology and notations that are not explained in the text, we refer to \cite{AB2003,AT2007}.

\section{Main results}

\hspace{4mm}
We start with the following theorem which asserts that collectively order-to-topology bounded sets 
quite often  agree with collectively \text{\rm ru}-to-topology continuous sets.

\begin{theorem}\label{theorem 2}
Let $X$ be an OVS and $(Y,\tau)$ a TVS. Then $\text{\bf L}_{o{\tau}b}(X,Y)\subseteq\text{\bf L}_{r{\tau}c}(X,Y)$.
If additionally $X_+$ is generating then $\text{\bf L}_{o{\tau}b}(X,Y)=\text{\bf L}_{r{\tau}c}(X,Y)$.
\end{theorem}

\begin{proof}
Assume, in contrary, ${\cal T}\in\text{\bf L}_{o{\tau}b}(X,Y)\setminus\text{\bf L}_{r{\tau}c}(X,Y)$.
Then, ${\cal T}x_\alpha\not\convctau 0$ for some $x_\alpha\convr 0$. So, there exists an absorbing $U\in\tau(0)$ such that, 
for every $\alpha$ there exist $\alpha'\ge\alpha$ and $T_\alpha\in{\cal T}$ with $T_\alpha x_{\alpha'}\notin U$.
Since $x_\alpha\convr 0$, for some $u\in X_+$ there exists an increasing sequence 
$(\alpha_n)$ of indices with $\pm nx_\alpha\le u$ for $\alpha\ge\alpha_n$. 
It follows from ${\cal T}\in\text{\bf L}_{o{\tau}b}(X,Y)$ that
${\cal T}[-u,u]\subseteq NU$ for some $N\in\mathbb{N}$.
Since $nx_\alpha\in[-u,u]$ for $\alpha\ge\alpha_n$ and $(\alpha_n)'\ge\alpha_n$ then
$nx_{(\alpha_n)'}\in[-u,u]$, and hence $T_{\alpha_n}(nx_{(\alpha_n)'})\in NU$ for every $n$.
In particular, $T_{\alpha_N}(x_{(\alpha_N)'})\in U$ which is absurd.
We conclude $\text{\bf L}_{o{\tau}b}(X,Y)\subseteq\text{\bf L}_{r{\tau}c}(X,Y)$.

Now, suppose $X_+$ is generating, and let ${\cal T}\in\text{\bf L}_{r{\tau}c}(X,Y)$. Assume, in contrary, ${\cal T}\notin\text{\bf L}_{o{\tau}b}(X,Y)$.
Since $X=X_+-X_+$, there exist $x\in X_+$ and an absorbing $U\in\tau(0)$ with ${\cal T}[-x,x]\not\subseteq nU$ for every $n\in\mathbb{N}$.  
Find sequences  $(x_n)$ in $[-x,x]$ and $(T_n)$ in ${\cal T}$ with $T_nx_n\notin nU$ for all $n$.
Since $\frac{1}{n}x_n\convr 0$ then ${\cal T}(\frac{1}{n}x_n)\convctau 0$.
So, there exists a sequence $(n_k)$ such that $T(\frac{1}{n}x_n)\in U$ for all $n\ge n_k$ and $T\in{\cal T}$. 
Then $T_{n_1}x_{n_1}\in n_1U$, which is a contradiction. Therefore, ${\cal T}\in\text{\bf L}_{o{\tau}b}(X,Y)$.
The proof is complete. 
\end{proof}

\begin{corollary}\label{cor 1 to theorem 2}
Every order-to-topology bounded operator from an OVS $X$ to a TVS $(Y,\tau)$ is ru-to-topology continuous.
Assuming in addition that $X_+$ is generating, ${\cal L}_{o{\tau}b}(X,Y)={\cal L}_{r{\tau}c}(X,Y)$.
\end{corollary}

The following result may be viewed as a topological version of \cite[Theorem 2.1]{EEG2025} 
(in the vector lattice setting see also \cite[Lemma 1.72]{AB2003}, \cite[Theorem 2.1]{AS2005}, and \cite[Theorem 2.1]{E0-2025}).

\begin{theorem}\label{theorem 2a}
Let $X$ be an Archimedean OVS with a generating cone and $(Y,\tau)$ a TVS. 
Then $\text{\bf L}_{o{\tau}c}(X,Y)\subseteq\text{\bf L}_{o{\tau}b}(X,Y)$.
\end{theorem}

\begin{proof}
We argue to a contradiction supposing ${\cal T}\in\text{\bf L}_{o{\tau}c}(X,Y)\setminus\text{\bf L}_{o{\tau}b}(X,Y)$.
Theorem \ref{theorem 2} implies ${\cal T}\notin\text{\bf L}_{r{\tau}c}(X,Y)$.
So, ${\cal T}x_\alpha\not\convctau 0$ for some net $(x_\alpha)$ in $X$ such that $x_\alpha\convr 0$.
Since $X$ is Archimedean then $x_\alpha\convo 0$, and hence ${\cal T}x_\alpha\convctau 0$
because of ${\cal T}\in\text{\bf L}_{o{\tau}c}(X,Y)$. The obtained contradiction completes the proof.
\end{proof}

\begin{corollary}\label{cor 1 to theorem 2a}
Every order-to-topology continuous operator from an Archimedean OVS with a generating cone to a TVS is order-to-topology bounded.
\end{corollary}

It is well known that $\text{\bf L}_{c}(X,Y)=\text{\bf L}_{b}(X,Y)$ whenever $X$ and $Y$ are NSs. 
The following result can be viewed as a partial extension of \cite[Theorem 2.1]{E0-2025} and \cite[Theorem 2.8]{EEG2025}.

\begin{theorem}\label{theorem 1}
Let $X$ be an OBS with a closed generating cone and $(Y,\tau)$ a TVS. Then $\text{\bf L}_{o{\tau}b}(X,Y)\subseteq\text{\bf L}_{n{\tau}b}(X,Y)$.
\end{theorem}

\begin{proof}
Let ${\cal T}\in\text{\bf L}_{o{\tau}b}(X,Y)$. Suppose, in contrary, ${\cal T}\notin\text{\bf L}_{n{\tau}b}(X,Y)$.
By the Krein -- Smulian theorem (cf., \cite[Theorem 2.37]{AT2007})
$\alpha B_X\subseteq B_X\cap X_+-B_X\cap X_+$ for some $\alpha>0$,
and hence ${\cal T}(B_X\cap X_+)$ is not $\tau$-bounded. Then, there exists an absorbing $U\in\tau(0)$ with
${\cal T}(B_X\cap X_+)\not\subseteq nU$ for every $n\in\mathbb{N}$. So, for some  
sequences $(x_n)$ in $B_X\cap X_+$ and $(T_n)$ in ${\cal T}$ we have $T_nx_n\notin n^3U$ for all $n$.
Set $x:=\|\cdot\|$-$\sum\limits_{n=1}^\infty n^{-2}x_n\in X_+$.
Since ${\cal T}\in\text{\bf L}_{o{\tau}b}(X,Y)$ then ${\cal T}[0,x]\subseteq NU$ for some $N\in\mathbb{N}$.
We deduce from $n^{-2}x_n\in[0,x]$ that $T_n(n^{-2}x_n)\in NU\subseteq nU$ for large enough $n$.
This is absurd, because $T_n(n^{-2}x_n)\notin nU$ for every $n$. The proof is complete. 
\end{proof}

\noindent 
Norm completeness of $X$ is essential in Theorem \ref{theorem 1}. For example, an operator $T\in{\cal L}(c_{00})$
defined by $Tx=(\sum_{k=1}^\infty x_k)e_1$ is order-to-norm bounded, yet not bounded.
To see that the condition $X_+-X_+=X$ is essential, it is enough to take any unbounded operator $T$ on a Banach space
$X$ with a trivial cone $X_+=\{0\}$. The closeness of $X_+$ in a Banach space $X$ is also essential \cite[Example 2.12 b)]{EEG2025}.
The next corollary a kind of Uniform Boundedness Principle for
families of operators which need not to be continuous a priory.

\begin{corollary}\label{cor 1 to theorem 1}
Every collectively order-to-norm bounded set of operators from an OBS with a closed generating cone to a NS is uniformly bounded.
\end{corollary}

\begin{corollary}\label{cor 1a to theorem 1}
Every order-to-norm bounded operator from an OBS with a closed generating cone to a NS is bounded.
\end{corollary}

\noindent
The following corollary of Theorem \ref{theorem 1} provides an automatic continuity result extending the well known fact 
(see, for example \cite[Theorem 2.32]{AT2007}) that every positive operator from an OBS with a closes generating cone
to an OBS with a closed cone is continuous. 

\begin{corollary}\label{cor 3 to theorem 1}
Every order bounded operator from an OBS with a closed generating cone to an ONS with a normal cone is continuous.
\end{corollary}

\noindent
Since in NSs bounded operators agree with continuous, and since weak compact sets in Banach spaces are bounded, 
we obtain the following consequence of Theorem \ref{theorem 1}.

\begin{corollary}\label{cor 2 to theorem 1}
Every order-to-weak compact operator from an OBS with a closed generating cone to a Banach space is continuous.
\end{corollary}

\noindent
Collective boundedness of a semigroup generated by a single operator is known as power boundedness of the operator.
We say that an operator $T$ on an ordered TVS $(X,\tau)$ is power order-to-topology bounded if the set $\bigcup\limits_{n=1}^\infty T^n[a,b]$
is $\tau$-bounded for every $[a,b]$ in $X$. An operator semigroup ${\cal S}$ on an ordered TVS OVS $(X,\tau)$ is order-to-topology bounded 
if the set $\bigcup\limits_{T\in{\cal S}}T[a,b]$ is $\tau$-bounded for every $[a,b]$ in $X$.
The next two corollaries deal with these notions.

\begin{corollary}\label{cor 4 to theorem 1}
Every power order-to-norm bounded operator on an OBS with a closed generating cone is power bounded.
\end{corollary}

\begin{corollary}\label{cor 5 to theorem 1}
Every power order-to-norm bounded operator semigroup on an OBS with a closed generating cone is uniformly continuous.
\end{corollary}

Since $\text{\bf L}_{b}(X,Y)\subseteq\text{\bf L}_{onb}(X,Y)$ whenever $X$ is a normal OVS and $Y$ is a NS, 
the next collective extension of \cite[Proposition 1.5]{E1-2025} follows from Theorems \ref{theorem 1} and \ref{theorem 2}.

\begin{proposition}\label{theorem 3}
Let $X$ be an OBS with a closed generating normal cone and $Y$ a NS. 
Then $\text{\bf L}_{rnc}(X,Y)=\text{\bf L}_{onb}(X,Y)=\text{\bf L}_{b}(X,Y)$.
\end{proposition}

\begin{corollary}\label{cor 1 to theorem 3}
Let $X$ be an OBS with a closed generating normal cone and $Y$ a NS. 
Then ${\cal L}_{rnc}(X,Y)={\cal L}_{onb}(X,Y)={\cal L}_{b}(X,Y)$.
\end{corollary}

\noindent
By \cite[Theorem 2.4]{EEG2025}, ${\cal L}_{rc}(X,Y)={\cal L}_{ob}(X,Y)$ whenever $X$ and $Y$ are OVSs with generating cones.
Therefore, we have one more consequence of Proposition \ref{theorem 3} that provides conditions for automatic boundedness of
ru-continuous operators.

\begin{corollary}\label{cor 2 to theorem 3}
Let $X$ be an OBS with a closed generating normal cone and $Y$ an OVS with a generating normal cone. 
Then ${\cal L}_{rc}(X,Y)\subseteq{\cal L}_{b}(X,Y)$.
\end{corollary}

The following proposition is a minor extension of \cite[Lemma 2.1]{E1-2025}.

\begin{proposition}\label{last}
Let $X$ be a normal OBS and $(Y,\tau)$ a Banach space with a dual topology. Then $\text{\rm L}_{o{\tau}c}(X,Y)\subseteq\text{\rm L}_{onb}(X,Y)$.
\end{proposition}

\begin{proof}
Let $T:X\to Y$ be \text{\rm o$\tau$}-continuous. Suppose, in contrary, $T\notin\text{\rm L}_{onb}(X,Y)$.
Then $T[0,u]$ is not bounded for some $u\in X_+$. 
Since $X$ is normal, $[0,u]$ is bounded, say $\sup\limits_{x\in[0,u]}\|x\|\le M$. 
Pick a sequence $(u_n)$ in $[0,u]$ with $\|Tu_n\|\ge n2^n$, and set $y_n:=\|\cdot\|$-$\sum\limits_{k=n}^\infty 2^{-k}u_k$ for $n\in\mathbb{N}$.
Then $y_n\downarrow\ge 0$. Let $0\le y_0\le y_n$ for all $n\in\mathbb{N}$.
Since $0\le y_0\le 2^{1-n}u$ and $\|2^{1-n}u\|\le M2^{1-n}\to 0$ then $y_0=0$ by \cite[Theorem 2.23]{AT2007}.
Since $y_n\downarrow 0$ then $y_n\convo 0$. We deduce from $T\in\text{\rm L}_{o{\tau}c}(X,Y)$ that $Ty_n\convtau 0$.
Since the topology $\tau$ is dual then $Ty_n\convw 0$.
Therefore, the sequence $(T_ny_n)$ is norm bounded. This is absurd because
$\|Ty_{n+1}-Ty_n\|=\|T(2^{-n}u_n)\|\ge n$ for all $n\in\mathbb{N}$. The proof is complete.
\end{proof}

\begin{corollary}\label{cor 1 to last}
Every order-to-topology continuous operator from a normal OBS to a Banach space with a dual topology is order-to-norm bounded.
\end{corollary}

\noindent
Combining Corollaries \ref{cor 1 to last} and \ref{cor 1a to theorem 1}, we obtain the last result of the paper.

\begin{corollary}\label{cor 2 to last}
Every order-to-topology continuous operator from an OBS with a closed generating normal cone to a Banach space with a dual topology is bounded.
\end{corollary}

\smallskip
{}
\end{document}